\documentclass[12pt,oneside,english]{amsart}
\usepackage[latin9]{inputenc}
\usepackage{geometry}
\usepackage{amstext}
\usepackage{amsthm}
\usepackage{amssymb}

\makeatletter
\numberwithin{equation}{section}
\numberwithin{figure}{section}
\theoremstyle{plain}
\newtheorem{thm}{\protect\theoremname}[section]
  \theoremstyle{plain}
  \newtheorem{conjecture}[thm]{\protect\conjecturename}
  \theoremstyle{definition}
  \newtheorem{defn}[thm]{\protect\definitionname}
  \theoremstyle{remark}
  \newtheorem{rem}[thm]{\protect\remarkname}
 \theoremstyle{definition}
 \newtheorem*{defn*}{\protect\definitionname}
  \theoremstyle{plain}
  \newtheorem{lem}[thm]{\protect\lemmaname}
  \theoremstyle{plain}
  \newtheorem{prop}[thm]{\protect\propositionname}
  \theoremstyle{plain}
  \newtheorem{cor}[thm]{\protect\corollaryname}
  \theoremstyle{definition}
  \newtheorem{example}[thm]{\protect\examplename}


\usepackage{yhmath} 

\makeatother

\usepackage{babel}
  \providecommand{\conjecturename}{Conjecture}
  \providecommand{\corollaryname}{Corollary}
  \providecommand{\definitionname}{Definition}
  \providecommand{\examplename}{Example}
  \providecommand{\lemmaname}{Lemma}
  \providecommand{\propositionname}{Proposition}
  \providecommand{\remarkname}{Remark}
\providecommand{\theoremname}{Theorem}

\begin{document}

\title{Vojta's conjecture for singular varieties}

\author{Takehiko Yasuda}
\begin{abstract}
We formulate a generalization of Vojta's conjecture in terms of log
pairs and variants of multiplier ideals. In this generalization, a
variety is allowed to have singularities. It turns out that the generalized
conjecture for a log pair is equivalent to the original conjecture
applied to a log resolution of the pair. A special case of the generalized
conjecture can be interpreted as representing a general phenomenon
that there tend to exist more rational points near singular points
than near smooth points. The same phenomenon is also observed in relation
between greatest common divisors of integer pairs satisfying an algebraic
equation and plane curve singularities, which is discussed in Appendix.
As an application of the generalization of Vojta's conjecture, we
also derive a generalization of a geometric conjecture of Lang concering
varieties of general type to singular varieties and log pairs.
\end{abstract}

\address{Department of Mathematics, Graduate School of Science, Osaka University,
Toyonaka, Osaka 560-0043, Japan, tel:+81-6-6850-5326, fax:+81-6-6850-5327}

\email{takehikoyasuda@math.sci.osaka-u.ac.jp, highernash@gmail.com}


\thanks{This work was supported by JSPS KAKENHI Grant Numbers JP15K17510
and JP16H06337. }

\maketitle
\global\long\def\AA{\mathbb{A}}
\global\long\def\PP{\mathbb{P}}
\global\long\def\NN{\mathbb{N}}
\global\long\def\GG{\mathbb{G}}
\global\long\def\ZZ{\mathbb{Z}}
\global\long\def\QQ{\mathbb{Q}}
\global\long\def\CC{\mathbb{C}}
\global\long\def\FF{\mathbb{F}}
\global\long\def\LL{\mathbb{L}}
\global\long\def\RR{\mathbb{R}}
\global\long\def\MM{\mathbb{M}}
\global\long\def\SS{\mathbb{S}}

\global\long\def\bx{\boldsymbol{x}}
\global\long\def\by{\boldsymbol{y}}
\global\long\def\bf{\mathbf{f}}
\global\long\def\ba{\mathbf{a}}
\global\long\def\bs{\mathbf{s}}
\global\long\def\bt{\mathbf{t}}
\global\long\def\bw{\mathbf{w}}
\global\long\def\bb{\mathbf{b}}
\global\long\def\bv{\mathbf{v}}
\global\long\def\bp{\mathbf{p}}
\global\long\def\bq{\mathbf{q}}
\global\long\def\bm{\mathbf{m}}
\global\long\def\bj{\mathbf{j}}
\global\long\def\bM{\mathbf{M}}
\global\long\def\bd{\mathbf{d}}

\global\long\def\cN{\mathcal{N}}
\global\long\def\cW{\mathcal{W}}
\global\long\def\cY{\mathcal{Y}}
\global\long\def\cM{\mathcal{M}}
\global\long\def\cF{\mathcal{F}}
\global\long\def\cX{\mathcal{X}}
\global\long\def\cE{\mathcal{E}}
\global\long\def\cJ{\mathcal{J}}
\global\long\def\cO{\mathcal{O}}
\global\long\def\cD{\mathcal{D}}
\global\long\def\cZ{\mathcal{Z}}
\global\long\def\cR{\mathcal{R}}
\global\long\def\cC{\mathcal{C}}
\global\long\def\cL{\mathcal{L}}
\global\long\def\cV{\mathcal{V}}
\global\long\def\cI{\mathcal{I}}
\global\long\def\cH{\mathcal{H}}
\global\long\def\cK{\mathcal{K}}

\global\long\def\fs{\mathfrak{s}}
\global\long\def\fp{\mathfrak{p}}
\global\long\def\fm{\mathfrak{m}}
\global\long\def\fX{\mathfrak{X}}
\global\long\def\fV{\mathfrak{V}}
\global\long\def\fx{\mathfrak{x}}
\global\long\def\fv{\mathfrak{v}}
\global\long\def\fY{\mathfrak{Y}}
\global\long\def\fa{\mathfrak{a}}
\global\long\def\fb{\mathfrak{b}}
\global\long\def\fc{\mathfrak{c}}

\global\long\def\rv{\mathbf{\mathrm{v}}}
\global\long\def\rx{\mathrm{x}}
\global\long\def\rw{\mathrm{w}}
\global\long\def\ry{\mathrm{y}}
\global\long\def\rz{\mathrm{z}}
\global\long\def\bv{\mathbf{v}}
\global\long\def\bw{\mathbf{w}}
\global\long\def\sv{\mathsf{v}}
\global\long\def\sx{\mathsf{x}}
\global\long\def\sw{\mathsf{w}}

\global\long\def\Spec{\mathrm{Spec}\,}
\global\long\def\Hom{\mathrm{Hom}}

\global\long\def\Var{\mathrm{Var}}
\global\long\def\Gal{\mathrm{Gal}}
\global\long\def\Jac{\mathrm{Jac}}
\global\long\def\Ker{\mathrm{Ker}}
\global\long\def\Image{\mathrm{Im}}
\global\long\def\Aut{\mathrm{Aut}}
\global\long\def\st{\mathrm{st}}
\global\long\def\diag{\mathrm{diag}}
\global\long\def\characteristic{\mathrm{char}}
\global\long\def\tors{\mathrm{tors}}
\global\long\def\sing{\mathrm{sing}}
\global\long\def\red{\mathrm{red}}
\global\long\def\ord{\mathrm{ord}}
\global\long\def\pt{\mathrm{pt}}
\global\long\def\op{\mathrm{op}}
\global\long\def\Val{\mathrm{Val}}
\global\long\def\Res{\mathrm{Res}}
\global\long\def\Pic{\mathrm{Pic}}
\global\long\def\disc{\mathrm{disc}}
\global\long\def\height{\mathrm{ht}}
 \global\long\def\length{\mathrm{length}}
\global\long\def\sm{\mathrm{sm}}
\global\long\def\rank{\mathrm{rank}}
\global\long\def\age{\mathrm{age}}
\global\long\def\et{\mathrm{et}}
\global\long\def\hom{\mathrm{hom}}
\global\long\def\tor{\mathrm{tor}}
\global\long\def\reg{\mathrm{reg}}
\global\long\def\cont{\mathrm{cont}}
\global\long\def\crep{\mathrm{crep}}
\global\long\def\Stab{\mathrm{Stab}}
\global\long\def\discrep{\mathrm{discrep}}
\global\long\def\totaldiscrep{\mathrm{totaldiscrep}}
\global\long\def\mld{\mathrm{mld}}

\global\long\def\GL{\mathrm{GL}}
\global\long\def\codim{\mathrm{codim}}
\global\long\def\Val{\mathrm{Val}}
\global\long\def\ur{\mathrm{ur}}
\global\long\def\cHom{\mathcal{H}om}
\global\long\def\cSpec{\mathcal{S}pec}
\global\long\def\Proj{\mathrm{Proj}\,}
\global\long\def\fie{\textrm{-}\mathrm{fie}}
\global\long\def\NS{\mathrm{NS}}
\global\long\def\Disc{\mathrm{Disc}}
\global\long\def\Kbar{\overline{K}}
\global\long\def\barK{\overline{K}}
\global\long\def\id{\mathrm{id}}
\global\long\def\Supp{\mathrm{Supp}}
\global\long\def\Exc{\mathrm{Exc}}
\global\long\def\Ram{\mathrm{Ram}}
\global\long\def\NonLC{\mathrm{NonLC}}
\global\long\def\NonSC{\mathrm{NonSC}}
\global\long\def\NonC{\mathrm{NonC}}
\global\long\def\NonKLT{\mathrm{NonKLT}}
\global\long\def\center{\mathrm{center}}
\global\long\def\div{\mathrm{div}}
\global\long\def\gcdn{\mathrm{gcdn}}
\global\long\def\Sp{\mathrm{Sp}}
\global\long\def\mult{\mathrm{mult}}

\tableofcontents{}

\section{Introduction}

Vojta's famous conjecture in the Diophantine geometry was originally
stated for a smooth variety $X$ and a simple normal crossing divisor
$D$ of it. In a recent paper \cite{MR2867936}, he generalized it
to an arbitrary proper closed subscheme $D\subset X$. The aim of
the present paper is to generalize it further by allowing $X$ to
have mild singularities. Our formulation is made in terms of log pairs
and their singularities, which are basic notions in the birational
geometry, in particular, in the minimal model program. We then show
that the generalized conjecture is in fact equivalent to the original
one in a similar way as Vojta did in the cited paper. We hope that
our formulation will encourage interaction between the Diophantine
geometry and the minimal model program.

Let $X$ be a smooth variety over a number field $k$ and $D$ a simple
normal crossing divisor of it. The original conjecture of Vojta concerns
the inequality 
\begin{equation}
h_{K_{X}}(x)-m_{D}(x)\le\epsilon h_{A}(x)+d_{k}(x)+O(1).\label{eq:ineq-intro-orig}
\end{equation}
See Sections \ref{sec:Weil functions} and \ref{sec:conjectures}
for details. Recently he generalized the conjecture to an aribitrary
proper closed subscheme $D\subset X$. Let $\cI\subset\cO_{X}$ the
multiplier ideal sheaf of $(X,(1-\epsilon)D)$ for sufficiently small
$\epsilon>0$, which was denoted by $\cJ^{-}$ in \cite{MR2867936}.
According to Silverman \cite{MR919501}, we can define proximity functions
$m_{W}$ and counting functions $N_{W}$ of closed subschemes $W\subset X$
and similarly ones of ideal sheaves. In \cite{MR2867936}, Vojta generalized
the conjecture by changing the left hand side of (\ref{eq:ineq-intro-orig})
to
\[
h_{K_{X}}(x)-m_{D}(x)-m_{\cI}(x),
\]
introducing the correction term $-m_{\cI}(x)$. 

In this paper, we further generalize it, allowing the variety $X$
to have (not necessarily normal) $\QQ$-Gorenstein singularities and
$D$ to be $\QQ_{\ge0}$-linear combination of closed subschemes.
We can similarly define the ideal sheaf $\cI$ also in this case and
define another ideal sheaf $\cH$, which is again a variant of the
multiplier ideal. We change the left hand side of the inequality to
\begin{equation}
h_{K_{X}}(x)+h_{D}(x)-N_{\cH}(x)-m_{\cI}(x).\label{eq:intro-gen}
\end{equation}
If $X$ is smooth, then we have $m_{D}\le h_{D}-N_{\cH}$ and our
inequality is slightly stronger than the one of Vojta. However it
turns out that Vojta's and our generalizations are eventually equivalent
to the original conjecture. The first two terms $h_{K_{X}}+h_{D}$
of (\ref{eq:intro-gen}) is then interpreted as the height function
of the ``canonical divisor'' $K_{(X,D)}=K_{X}+D$ of the log pair
$(X,D)$, and the last two terms $-N_{\cH}(x)-m_{\cI}(x)$ as contribution
of singularities of $(X,D)$. In the special case where $D=0$ and
$X$ has only log terminal singularities, then (\ref{eq:intro-gen})
is written as
\[
h_{K_{X}}(x)-N_{\NonC(X)}(x),
\]
where $\NonC(X)$ is the non-canonical locus of $X$. As an example
showing the necessity of the correction term $-N_{\NonC(X)}(x)$,
we construct a ``rational surface of general type'' having only
quotient singularities (Section \ref{sec:Rational-surfaces-of-general-type}).

From our generalization of Vojta's conjecture, we derive the following
geometric conjecture:
\begin{conjecture}[See Corollary \ref{cor:singular-Lang-NonSC} for a more general version]
\label{conj:general-lang-intro}Let $X$ be a $\QQ$-Gorenstein variety
with a canonical divisor $K_{X}$ big. Then there exists a proper
closed subset $Z\subset X$ such that for every potentially dense
closed subvariety $Y\subset X$, either $Y$ is contained in $Z$
or $Y$ intersects $\NonC(X)$. 
\end{conjecture}
If $X$ is smooth, then $\NonC(X)$ is empty. The last conjecture
in this case was raised by Lang \cite{MR1112552}. 

As a height function $h_{D}(x)$ as well as counting and proximity
functions represents a closeness to $D$, we may regard our generalization
of Vojta's conjecture as a representation of a phenomenon that there
tend to be more rational points near singular points than near smooth
points. As another representation of the same phenomenon, in Appendix,
we discuss relation of greatest common divisors and plane curve singualrities.
Let $C\subset\PP_{k}^{2}$ be an irreducible plane curve of degree
$d$ and let $m$ be the multiplicity of $C$ at $O=(0:0:1)$. Let
$h_{O}$ be a height function of $O$ regarded as a closed reduced
subscheme of $\PP_{k}^{2}$ and $h$ be the standard height of $\PP_{k}^{2}$.
From the functoliality of height functions and a small computation
of resolution of singularities, we observe that when restricted to
$C(k)$, $h_{O}$ approximates $\frac{m}{d}h$. Combining it with
Silverman's interpretation of greatest common divisors in terms of
heights, we obtain a not deep but amusing fact that if $k=\QQ$, then
the greatest common divisor $\gcd(x,y)$ for $(x,y)\in C\cap\AA_{\QQ}^{2}$,
$x,y\in\ZZ$, approximates $\max\{|x|,|y|\}^{m/d}$ (for a more general
and precise statement, see Corollary \ref{cor:gcd}). 

Throughout the paper, we fix a number field $k$. A \emph{variety}
means a separated reduced scheme of pure dimension and finite type
over $k$. We suppose that every morphism of varieties is a morphism
of $k$-schemes. 

The author would like to thank Katsutoshi Yamanoi for helpful discussion,
which in particular led the author to Conjecture \ref{conj:general-lang-intro},
and Yu Yasufuku for useful information.

\section{\label{sec:Singularities-of-log-pairs}Singularities of log pairs}

In this section, we recall the notion of log pairs in a slightly generalized
context and basics of their singularities.
\begin{defn}
A variety $X$ is is said to be \emph{$\QQ$-Gorenstein }if
\begin{enumerate}
\item $X$ satisfies Serre's condition $S_{2}$,
\item $X$ is Gorenstein in codimension one, and
\item a canonical divisor $K_{X}$ is $\QQ$-Cartier.
\end{enumerate}
\end{defn}
Note that this definition of $\QQ$-Gorenstein is more general than
the usual one in the sense that we do not assume that $X$ is normal.
(However, if he prefers, the reader may safely assume that $X$ is
normal.) These conditions appear in the definition of ``pair'' in
\cite[Def. 1.5]{MR3057950} and the one of semi-log canonical singularities
in \cite[p. 35]{MR2675555}. From the first two conditions, which
are automatic if $X$ is normal, a canonical divisor $K_{X}$ exists,
is unique up to linear equivalence and is Cartier in codimension one.
Therefore the last condition makes sense and is equivalent to that
for some $m\in\ZZ_{>0}$, the reflexive power $\omega_{X}^{[m]}:=(\omega_{X}^{\otimes m})^{**}$
of the canonical sheaf $\omega_{X}$ is invertible. 

For a $\QQ$-Gorenstein variety $X$ and a proper birational morphism
$f\colon Y\to X$ of varieties with $Y$ normal, the pull-back $f^{*}K_{X}$
is defined as a $\QQ$-Cartier $\QQ$-Weil divisor of $Y$. 
\begin{defn}
A \emph{$\QQ$-subscheme} of a variety $X$ is a formal linear combination
$D=\sum_{i=1}^{n}c_{i}D_{i}$ of proper closed subschemes $D_{i}\subsetneq X$
with $c_{i}\in\QQ$. Following the terminology for divisors, we say
that a $\QQ$-subscheme $D=\sum_{i=1}^{n}c_{i}D_{i}$ is \emph{effective
}if every $c_{i}$ is non-negative; then we write $D\ge0$. The \emph{support}
of $D$, denoted $\Supp(D)$, is defined to be the closed subset $\left(\bigcup_{c_{i}\ne0}D_{i}\right)_{\red}$
and denoted by $\Supp(D)$. 

A \emph{log pair }is the pair $(X,D)$ of a $\QQ$-Gorenstein variety
$X$ and a $\QQ$-subscheme $D$ of $X$. We say that a log pair $(X,D)$
is \emph{effective }if $D$ is effective. We say that a log pair $(X,D)$
is \emph{projective }if $X$ is projective.
\end{defn}
When $D=0$, we usually omit $D$ from the notation and identify the
log pair $(X,0)$ with the variety $X$. For instance, the discrepancy
$\discrep(X,D)$ defined below will be written also as $\discrep(X)$
when $D=0$. 
\begin{rem}
If $X$ is a normal $\QQ$-Gorenstein variety and $D$ is a $\QQ$-Cartier
$\QQ$-Weil divisor, then $D$ is written as a $\QQ$-linear combination
$\sum_{i=1}^{n}b_{i}E_{i}$ of effective Cartier divisors $E_{i}$;
this allows us to regard the pair $(X,D)$ as a log pair in the sense
defined above. \end{rem}
\begin{defn*}
A \emph{resolution }of a variety $X$ is a proper birational morphism
$f\colon Y\to X$ such that $Y$ is smooth over $k$. Let $(X,D=\sum_{i=1}^{n}c_{i}D_{i})$
be a log pair. A \emph{log resolution} of $(X,D)$ is a resolution
$f\colon Y\to X$ of $X$ such that
\begin{enumerate}
\item for every $i$, the scheme-theoretic preimage $f^{-1}(D_{i})$ is
a Cartier divisor (that is, if $\cI_{D_{i}}$ is the defining ideal
sheaf of $D_{i}$, then the pull-back $f^{-1}\cI_{D_{i}}$ as an ideal
sheaf is locally principal),
\item if we denote by $\Exc(f)$ the exceptional set of $f$, then $\Exc(f)\cup\bigcup_{i=1}^{n}f^{-1}(D_{i})_{\red}$
is a simple normal crossing divisor.
\end{enumerate}
\end{defn*}
From Hironaka's theorem, any variety has a resolution and any log
pair has a log resolution. 
\begin{defn}
For a $\QQ$-Gorenstein variety and a resolution $f\colon Y\to X$
of $X$, the \emph{relative canonical divisor }$K_{Y/X}$ \emph{of
$Y$ over} $X$ is defined as a $\QQ$-divisor of $Y$ supported in
$\Exc(f)$ as follows. If $m$ is a positive integer suh that $\omega_{X}^{[m]}$
is invertible, then the natural morphism $f^{*}\omega_{X}^{[n]}\to\omega_{Y}^{\otimes n}$
is injective and its image is written as $\omega_{Y}^{\otimes n}(\Delta)$
for some ($\ZZ$-)divisor $\Delta$. We then define 
\[
K_{Y/X}:=-\frac{1}{n}\Delta.
\]

For a log pair $(X,D)$ and a log resolution $f\colon Y\to X$ of
it, we define the \emph{relative canonical divisor $K_{Y/(X,D)}$
of $Y$ over} $(X,D)$ as the $\QQ$-divisor $K_{Y/X}-f^{*}D$. Here,
if we write $D=\sum_{i=1}^{l}c_{i}D_{i}$, then we define the \emph{pull-back}
$f^{*}D$ as $\sum_{i=1}^{l}c_{i}f^{-1}D_{i}$, which is a $\QQ$-divisor.
Let us write 
\[
K_{Y/(X,D)}=\sum_{F}a_{F}\cdot F,
\]
$F$ running over the prime divisors of $Y$. We call $a_{F}$ the
\emph{discrepancy} of $F$ with respect to $(X,D)$ and write it as
$a(F;X,D)$. 
\end{defn}

\begin{defn}
Let $X$ be a variety. A \emph{divisor over }$X$ is a prime divisor
$F$ on $Y$ for a resolution $f\colon Y\to X$. We say that such
an $F$ is called \emph{exceptional }if $f$ is not an isomorphism
at the generic point of $F$. 

Let $(X,D)$ be a log pair. We define the \emph{discrepancy }of $(X,D)$
by
\[
\discrep(X,D):=\inf\{a(F;X,D)\mid F\text{ is an exceptional divisor over \ensuremath{X}}\}.
\]
and the \emph{total discrepancy }of $(X,D)$ by 
\[
\totaldiscrep(X,D):=\inf\{a(F;X,D)\mid F\text{ is a divisor over \ensuremath{X}}\}.
\]

\end{defn}
We are mainly interested in the total discrepancy rather than the
discrepancy. It is easy to see that for a log resolution $f\colon Y\to X$
of a log pair $(X,D)$, we have
\begin{equation}
\totaldiscrep(Y,-K_{Y/(X,D)})=\totaldiscrep(X,D).\label{eq:totaldiscrep-resol}
\end{equation}

\begin{lem}
We have either
\[
\totaldiscrep(X,D)=-\infty
\]
or
\[
-1\le\totaldiscrep(X,D)\le0.
\]
\end{lem}
\begin{proof}
Firstly, since a general prime divisor has zero discrepancy, we have
\[
\totaldiscrep(X,D)\le0.
\]
The lemma is well known in the case where $X$ is normal (see \cite[Cor. 2.31]{MR1658959}).
The general case follows from (\ref{eq:totaldiscrep-resol}).
\end{proof}
We can now define three classes of singularities of log pairs as follows:
\begin{defn}
We say that $(X,D)$ is \emph{strongly canonical} (resp. \emph{Kawamata
log terminal}, \emph{log canonical}) if 
\[
\mathrm{totaldiscrep}(X,D)\ge0\quad(\text{resp. }>-1,\,\ge-1).
\]
We say that $X$ is \emph{canonical }(resp. \emph{log terminal}, log
canonical) if $(X,0)$ is strongly canonical (resp. Kawamata log terminal,
log canonical). \end{defn}
\begin{rem}
The author does not know whether the notion of strongly canonical
has been ever considered, while the other notions are quite standard.
This notion is necessary for our reformulation of Vojta's conjecture.
A log pair $(X,D)$ is said to be \emph{canonical }if $\discrep(X,D)\ge0$;
this is standard, but we do not use it in this paper. When $D=0$,
strongly canonical and canonical are equivalent notions.
\end{rem}

\begin{rem}
\label{rem:coeff-sc-klt-lc}When $X$ is normal and $\QQ$-Gorenstein
and $D$ is a $\QQ$-Cartier $\QQ$-Weil divisor, then for a log pair
$(X,D)$ being strongly canonical (resp. Kawamata log terminal, log
canonical), the multiplicity of every prime divisor in $D$ needs
to be $\le0$ (resp. $<1$, $\le1$). 
\end{rem}
The definition of total discrepancy uses all divisors over a given
variety $X$. However the following lemma allows us to compute it
in terms of a single log resolution. 
\begin{lem}
\label{lem:totaldisc-log-resol}Let $X$ be a smooth variety and let
$D=\sum_{i=1}^{l}c_{i}F_{i}$ be a $\QQ$-divisor of $X$ such that
$c_{i}\in\QQ$, $F_{i}$ are prime divisors and $\bigcup_{i=1}^{l}F_{i}$
is a simple normal crossing divisor. Then $(X,D)$ is strongly canonical
(resp. Kawamata log terminal, log canonical) if and only if $c_{i}\le0$
(resp. $<1$, $\le1$) for every $i$. \end{lem}
\begin{proof}
If $c_{i}=-a(F_{i};X,D)>1$ for some $i$, then $\totaldiscrep(X,D)<-1$
and $(X,D)$ is not log canonical. Otherwise, from \cite[Cor. 2.11]{MR3057950},
\[
\discrep(X,D)=\min\{1,\min_{i}\{1-c_{i}\},\,\min_{F_{i}\cap F_{j}\ne\emptyset}\{1-c_{i}-c_{j}\}\}.
\]
The lemma follows from 
\[
\totaldiscrep(X,D)=\min\{0,\,\min_{i}\{-c_{i}\},\,\discrep(X,D)\}.
\]

\end{proof}
The above notions of singularities are local; $(X,D)$ is strongly
canonical (resp. Kawamata log terminal, log canonical) if and only
if every point $x\in X$ has an open neighborhood $U\subset X$ such
that $(U,D|_{U})$ is so. 
\begin{defn}
Let $(X,D)$ be a log pair. The \emph{non-sc locus} (resp. \emph{non-lc
locus}) of $(X,D)$ is the smallest closed subset $W\subset X$ such
that $(X\setminus W,D|_{X\setminus W})$ is strongly canonical (resp.
Kawamata log terminal, log canonical). We write it as $\NonSC(X,D)$
(resp. $\NonKLT(X,D)$, $\NonLC(X,D)$). We call the non-sc locus
$\NonSC(X,0)$ of the log pair $(X,0)$ also as the \emph{non-canonical
locus} of $X$ and denote it by $\NonC(X)$. 
\end{defn}
Clearly 
\[
\NonLC(X,D)\subset\NonKLT(X,D)\subset\NonSC(X,D).
\]

\begin{lem}
\label{lem:NonSC-NonLC-descrip}Let $f\colon Y\to X$ be a log resolution
of $(X,D)$. Then
\begin{gather*}
\NonSC(X,D)=\bigcup_{F\subset Y:a(F;X,D)<0}f(F),\\
\NonKLT(X,D)=\bigcup_{F\subset Y:a(F;X,D)\le-1}f(F)\\
\NonLC(X,D)=\bigcup_{F\subset Y:a(F;X,D)<-1}f(F),
\end{gather*}
in each of which $F$ runs over the prime divisors of $Y$ satisfying
the indicated inequality. \end{lem}
\begin{proof}
This is a direct consequence of Lemma \ref{lem:totaldisc-log-resol}.
\end{proof}

\section{\label{sec:Multiplier-like-ideals}Multiplier-like ideal sheaves}

In this section, we define two variants of multiplier ideals and study
their basic properties.

For an effective log pair $(X,D)$ with $X$ normal, the \emph{multiplier
ideal sheaf} $\cJ(X,D)$ is usually defined to be $f_{*}\cO_{Y}(\lceil K_{Y/(X,D)}\rceil)$
for a log resolution $f\colon Y\to X$ of $(X,D)$ (see \cite[Def. 9.3.56]{MR2095472}).
Here $\lceil\cdot\rceil$ denotes the round up, while $\lfloor\cdot\rfloor$
used below denotes the round down. When $X$ is not normal, $f_{*}\cO_{Y}(\lceil K_{Y/(X,D)}\rceil)$
is no longer an ideal sheaf. To handle this trouble, we replace $f_{*}$
with $f_{\clubsuit}$ defined as follows:
\begin{defn}
For a proper birational morphism $f\colon Y\to X$ of varieties and
a divisor $E$ of $Y$, we define $f_{\clubsuit}\cO_{Y}(E)$ as the
largest ideal sheaf $\cI\subset\cO_{X}$ such that the ideal pull-back
$f^{-1}\cI$ is contained in $\cO_{Y}(E)$ as an $\cO_{Y}$-submodule
of the sheaf of total quotient rings. 
\end{defn}
We then generalize the multiplier ideal sheaf to the non-normal case
as follows:
\begin{defn}
Let $(X,D)$ be an effective log pair and $f\colon Y\to X$ a log
resolution of it. The \emph{multiplier ideal sheaf }$\cJ(X,D)$ is
defined to be $f_{\clubsuit}\cO_{Y}(\lceil K_{Y/(X,D)}\rceil)$.
\end{defn}
We will define two variants $\cH(X,D)$ and $\cI(X,D)$ of the multiplier
ideal to formulate a generalization of Vojta's conjecture for log
pairs. We first define $\cH(X,D)$. 
\begin{lem}
Let $X$ be a smooth variety, $E$ a (not necessarily effective) $\QQ$-divisor
and $f\colon Y\to X$ a proper birational morphism. Then
\[
f_{*}\cO_{Y}(\left\lfloor K_{Y/X}+f^{*}E\right\rfloor )=\cO_{X}(\left\lfloor E\right\rfloor ).
\]
\end{lem}
\begin{proof}
First suppose that $\lfloor E\rfloor=0$. To show the lemma in this
case, it suffices to show that $\left\lfloor K_{Y/X}+f^{*}E\right\rfloor $
is an effective divisor supported in $\Exc(f)$. Since $K_{Y/X}$
and $E$ are effective, so is $\left\lfloor K_{Y/X}+f^{*}E\right\rfloor $.
On the locus where $f$ is an isomorphism, the two divisors $\left\lfloor K_{Y/X}+f^{*}E\right\rfloor $
and $\lfloor E\rfloor$ coincide, the latter being zero by the assumption.
This proves the lemma in this case.

For the general case, we write $\{E\}:=E-\lfloor E\rfloor$. Obviously,
$\lfloor\{E\}\rfloor=0$. From the projection formula and the case
considered above, we have
\begin{align*}
f_{*}\cO_{Y}(\left\lfloor K_{Y/X}+f^{*}E\right\rfloor ) & =f_{*}\cO_{Y}(\left\lfloor K_{Y/X}+f^{*}\{E\}\right\rfloor +f^{*}\lfloor E\rfloor)\\
 & =f_{*}\left(\cO_{Y}(\left\lfloor K_{Y/X}+f^{*}\{E\}\right\rfloor )\otimes_{\cO_{Y}}f^{*}\cO_{X}(\lfloor E\rfloor)\right)\\
 & =\cO_{X}\otimes_{\cO_{X}}\cO_{X}(\lfloor E\rfloor)\\
 & =\cO_{X}(\lfloor E\rfloor).
\end{align*}
We have completed the proof.\end{proof}
\begin{prop}
Let $(X,D)$ be an effective log pair and $f\colon Y\to X$ a log
resolution of $(X,D)$. Then $f_{\clubsuit}\cO_{Y}\left(\lfloor K_{Y/(X,D)}\rfloor\right)$
is an ideal sheaf contained in $\cJ(X,D)$ and independent of $f$.\end{prop}
\begin{proof}
Since $\left\lfloor K_{Y/(X,D)}\right\rfloor \le\left\lceil K_{Y/(X,D)}\right\rceil $,
we have 
\[
\cO_{Y}(\lfloor K_{Y/(X,D)}\rfloor)\subset\cO_{Y}(\lceil K_{Y/(X,D)}\rceil).
\]
Applying $f_{\clubsuit}$, we obtain the first assertion.

To show the second assertion, we consider another log resolution $f'\colon Y'\to X$
of $(X,D)$. Without loss of generality, we may suppose that $f'$
factors as $f\circ g$ with a morphism $g\colon Y'\to Y$. Then 
\[
K_{Y'/(X,D)}=K_{Y'/Y}+g^{*}K_{Y/(X,D)}.
\]
From the above lemma, 
\begin{align*}
f'_{\clubsuit}\cO_{Y'}(\lfloor K_{Y'/(X,D)}\rfloor) & =f_{\clubsuit}\left(g_{*}\cO_{Y'}(\lfloor K_{Y'/Y}+g^{*}K_{Y/(X,D)}\rfloor)\right)\\
 & =f_{\clubsuit}\cO_{Y}(\lfloor K_{Y/(X,D)}\rfloor).
\end{align*}
We have proved the second assertion. \end{proof}
\begin{defn}
For an effective log pair $(X,D)$, we define an ideal sheaf $\cH(X,D)$
on $X$ by 
\[
\cH(X,D):=f_{\clubsuit}\cO_{Y}(\lfloor K_{Y/(X,D)}\rfloor)
\]
for a log resolution $f\colon Y\to X$ of $(X,D)$. \end{defn}
\begin{prop}
\label{prop:H-sc}For an effective log pair $(X,D)$ and a point $x\in X$,
we have $\cH(X,D)_{x}=\cO_{X,x}$ if and only if $(X,D)$ is strongly
canonical around $x$. Equivalently, the support of 
\[
\Supp(\cO_{X}/\cH(X,D))=\NonSC(X,D).
\]
\end{prop}
\begin{proof}
We write $\cH(X,D)$ as $\cH$. We first show the ``if'' part. We
suppose that $(X,D)$ is strongly canonical. Let $f\colon Y\to X$
be a log resolution of $(X,D)$. By the definition of strongly canonical,
$K_{Y/(X,D)}\ge0$ and $\lfloor K_{Y/(X,D)}\rfloor\ge0$. By the definition
of $f_{\clubsuit}$, we have
\[
\cH=f_{\clubsuit}\cO_{Y}(\lfloor K_{Y/(X,D)}\rfloor)=\cO_{X}.
\]

Next we show the ``only if'' part. Suppose that $(X,D)$ is not
strongly canonical around $x$. Then there exists a prime divisor
$F$ on $Y$ such that $x\in f(F)$ and the multiplicity of $F$ in
$K_{Y/(X,D)}$ is negative. Therefore, 
\[
f^{-1}\cO_{X}=\cO_{Y}\not\subset\cO_{Y}(\lfloor K_{Y/(X,D)}\rfloor).
\]
This remains true even if we replace $X$ with any open neighborhood
of $x$. Agin by the definition of $f_{\clubsuit}$, $\cH_{x}\not\ne\cO_{X,x}$. \end{proof}
\begin{cor}
\label{cor:H-def-ideal}Let $(X,D)$ be an effective log pair. If
$(X,D)$ is log canonical, then $\cH(X,D)$ is the defining ideal
sheaf of the closed subset $\NonSC(X,D)$. \end{cor}
\begin{proof}
Let $\cN$ be the defining ideal of $\NonSC(X,D)$. From Proposition
\ref{prop:H-sc}, we have $\cH\subset\cN$. To see the opposite inclusion,
let $U\subset X$ be an open subset and $g\in\cN(U)$. For a log resolution
$f\colon Y\to X$ of $(X,D)$, $f^{*}g$ vanishes along the closed
set $f^{-1}(\NonSC(X,D))$. The last set contains every prime divisor
$F$ on $Y$ having a negative coefficient in $\lfloor K_{Y/(X,D)}\rfloor$,
which is equal to $-1$ since $(X,D)$ is log canonical. Therefore,
$f^{*}g\in\cO_{Y}(\lfloor K_{Y/(X,D)}\rfloor)(f^{-1}U)$ and hence
$g\in\cH(U)$. Thus $\cN\subset\cH$, proving the corollary. 
\end{proof}
Next we will define the other variant, denoted by $\cI(X,D)$, of
the multiplier ideal.
\begin{lem}
Let $(X,D)$ be an effective log pair. Suppose that $D$ is effective.
There exists a positive rational number $\epsilon_{0}$ such that
for every $\epsilon\in(0,\epsilon_{0}]\cap\QQ$, 
\[
\cJ(X,(1-\epsilon)D)=\cJ(X,(1-\epsilon_{0})D).
\]
\end{lem}
\begin{proof}
Let $f\colon Y\to X$ be a log resolution of $(X,D)$. For a rational
number $\epsilon>0$, 
\[
K_{Y/(X,(1-\epsilon)D)}=K_{Y/(X,D)}+\epsilon f^{*}D.
\]
We choose so small $\epsilon_{0}>0$ that every coefficient of $\epsilon_{0}f^{*}D$
is smaller than the fractional par $\{x\}:=x-\lfloor x\rfloor$ of
any non-integral coefficient $x$ of $K_{Y/(X,D)}$. Then, for every
$\epsilon\in(0,\epsilon_{0}]$, 
\[
\left\lceil K_{Y/(X,(1-\epsilon)D)}\right\rceil =\left\lceil K_{Y/(X,(1-\epsilon_{0})D)}\right\rceil ,
\]
which shows the lemma. \end{proof}
\begin{defn}
For an effective log pair $(X,D)$, we define an ideal sheaf $\cI(X,D)$
on $X$ as the multiplier ideal sheaf $\cJ(X,(1-\epsilon)D)$ for
a sufficiently small rational number $\epsilon>0$. \end{defn}
\begin{rem}
\label{rem:compare-ideals}Let $\fa\subset\cO_{X}$ be the defining
ideal sheaf of $D$. When $X$ is smooth, then the ideal $\cI(X,D)$
was denoted by $\cJ^{-}(\fa)$ in \cite{MR2867936} and Vojta used
it in a generalization of his own conjecture. 
\end{rem}
Multiplier-like ideal sheaves which we saw above satisfy the following
inclusion relations,
\[
\cH(X,D)\subset\cJ(X,D)\subset\cI(X,D).
\]

\begin{prop}[{cf. \cite[Def. 9.3.9]{MR2095472}}]
\label{prop:non-lc-locus}Let $(X,D)$ be an effective log pair.
Then
\begin{equation}
\NonLC(X,D)\subset\Supp(\cO_{X}/\cI(X,D))\subset\NonKLT(X,D).\label{eq:lc-klt}
\end{equation}
Moreover, if $X$ is log terminal outside $\Supp(D)$, then 
\begin{equation}
\NonLC(X,D)=\Supp(\cO_{X}/\cI(X,D)).\label{eq:lc}
\end{equation}
\end{prop}
\begin{proof}
Let $f\colon Y\to X$ be a log resolution of $(X,D)$ and $\epsilon>0$
a sufficiently small rational number. The closed subset defined by
$\cI(X,D)$ is expressed as 
\[
\bigcup_{F\subset Y:\mult_{F}(\lceil K_{Y/(X,D)}+\epsilon f^{*}D\rceil)<0}f(F),
\]
where $\mult_{F}(E)$ denotes the multiplicity of $F$ in $E$, while
$\NonKLT(X,D)$ and $\NonLC(X,D)$ have similar expressions as in
Lemma \ref{lem:NonSC-NonLC-descrip}. We have the following implications
among conditions on a prime divisor $F$ of $Y$,
\[
a(F;X,D)<-1\Rightarrow\mult_{F}(\lceil K_{Y/(X,D)}+\epsilon f^{*}D\rceil)<0\Rightarrow a(F;X,D)\le-1.
\]
This shows (\ref{eq:lc-klt}).

To show (\ref{eq:lc}), it suffices to show 
\[
a(F;X,D)\ge-1\Rightarrow\mult_{F}(\lceil K_{Y/(X,D)}+\epsilon f^{*}D\rceil)\ge0.
\]
We consider the case $a(F;X,D)>-1$ and the case $a(F;X,D)=-1$ separately.
In the former, we obviously have $\mult_{F}(\lceil K_{Y/(X,D)}+\epsilon f^{*}D\rceil)\ge0$.
In the latter, since $X$ is log terminal outside $\Supp(D)$, $F$
is contained in $\Supp(f^{*}D)$. Hence $\mult_{F}(K_{Y/(X,D)}+\epsilon f^{*}D)>-1$
and $\mult_{F}(\left\lceil K_{Y/(X,D)}+\epsilon f^{*}D\right\rceil )\ge0$.
This completes the proof. 
\end{proof}

\section{\label{sec:Weil functions}Weil functions}

In this section, we summerize basic properties of Weil functions (local
height functions) of arbitrary closed subschemes, and define associated
heigh functions, counting functions and proximity functions. 

We denote by $M_{k}$ the set of places of $k$. In what follows,
we fix a finite set $S\subset M_{k}$ containing all infinite places.
We also fix an algebraic closure $\bar{k}$ of $k$. 

Let $X$ be a projective variety. To an ideal sheaf $\fa\subset\cO_{X}$,
we associate a Weil function
\[
\lambda_{\fa}\colon X(\bar{k})\times M_{k}\to[0,+\infty],
\]
following \cite{MR919501}, which is unique up to addition of $M_{k}$-bounded
functions. If $Z$ is the closed subscheme defined by an ideal sheaf
$\fa$, then we write $\lambda_{\fa}$ also as $\lambda_{Z}$. If
$Z$ is a Cartier divisor (that is, $\fa$ is locally principal),
then it is the usual Weil function for an effective Cartier divisor. 

There are several ways to normalize Weil functions. We follow the
one adopted in \cite[Def. 8.6]{MR2757629}. Namely, for $v\in M_{k}$,
if $p$ is the place of $\QQ$ under $v$, then we denote by $\Vert\cdot\Vert_{v}$
the norm on $\bar{k}$ extending the one on $k$ defined by 
\[
\Vert a\Vert_{v}:=|\cdot|_{p}^{[k_{v}:\QQ_{p}]},
\]
where $|\cdot|_{p}$ denotes the usual $p$-adic absolute value and
$k_{v}$ denotes the $v$-adic completion of $k$. When a Cartier
divisor $D$ is locally defined by a rational function $f$, then
a Weil function $\lambda_{D}$ of $D$ should be locally of the form
\[
\lambda_{D}(x,v)=-\log\Vert f(x)\Vert_{v}+\alpha(x)
\]
for a continuous locally $M_{k}$-bounded function $\alpha$. 

Basic properties of Weil functions are as follows; in this proposition,
comparison (equality or inequality) of Weil functions are made up
to addition of $M_{k}$-bounded functions:
\begin{prop}[{\cite[Th. 2.1]{MR919501}}]
\label{prop:properties-Weil-fn}
\begin{enumerate}
\item For a morphism $f\colon Y\to X$ of varieties and a closed subvariety,
we have
\[
\lambda_{Z}\circ f=\lambda_{f^{-1}Z}.
\]

\item For $Z\subset Z'\subset X$, 
\[
\lambda_{Z}\le\lambda_{Z'}.
\]

\item For closed subvarieties $Z,Z'\subset X$,
\[
\lambda_{Z+Z'}=\lambda_{Z}+\lambda_{Z'}.
\]
Here, if $Z$ and $Z'$ are defined by ideal sheaves $\fa$ and $\fa'$
respectively, then $Z+Z'$ is the closed subscheme defined by the
product $\fa\fa'$.
\item \label{enu:min-lambda}For closed subvarieties $Z,Z'\subset X$,
\[
\lambda_{Z\cap Z'}=\min\{\lambda_{Z}+\lambda_{Z'}\}.
\]
Here, if $Z$ and $Z'$ are defined by ideal sheaves $\fa$ and $\fa'$
respectively, then $Z\cap Z'$ is the closed subscheme defined by
the sum $\fa+\fa'$.
\end{enumerate}
\end{prop}
For later use, we need the following explicit description of a Weil
function in a special case:
\begin{prop}[{cf. \cite[Th. 8.8, (c)]{MR2757629}, \cite[Ex. B.8.4]{MR1745599}}]
\label{prop:normalization-Weil-fn}Let $X=\PP_{k}^{n}$ be a projective
space of dimension $n$ with homogeneous coordinates $x_{0},\dots,x_{n}$
and $D$ the Cartier divisor defined by a homogeneous polynomial $f\in k[x_{0},\dots,x_{n}]$
of degree $d$. Then the function 
\[
\lambda_{D_{i}}((x_{0}:\cdots:x_{n}),v):=-\log\frac{\Vert f(x_{0},\dots,x_{n})\Vert_{v}}{\max\{\Vert x_{0}\Vert_{v},\dots,\Vert x_{n}\Vert_{v}\}^{d}}
\]
is a Weil function with respect to $D$. 
\end{prop}
For $v\in M_{k}$ and $x\in X(\bar{k})$, we write 
\[
\lambda_{\fa,v}(x):=\lambda_{\fa}(x,v).
\]
For a finite extension $L/k$ and a place $w$ of $L$, if $v\in M_{k}$
is the place under $w$, then we define 
\[
\lambda_{\fa,w}(x):=[L_{w}:k_{v}]\cdot\lambda_{\fa,v}(x).
\]

\begin{defn}
We define the \emph{height function} $h_{\fa}$, the \emph{counting
function} $N_{\fa}$ and the \emph{proximity function} $m_{\fa}$
on $X(\bar{k})$ relative to $\lambda$ and $k$ as follows. For $x\in X(\bar{k})$,
we denote by $k(x)$ its residue field. Let $L/k$ be a finite extension
containing $k(x)$ and let $T\subset M_{L}$ be the set of places
over places in $S$. Then,
\begin{gather*}
h_{\fa}(x):=\frac{1}{[L:k]}\sum_{w\in M_{L}}\lambda_{\fa,w}(x),\\
N_{\fa}(x):=\frac{1}{[L:k]}\sum_{w\in M_{L}\setminus T}\lambda_{\fa,w}(x),\\
m_{\fa}(x):=\frac{1}{[L:k]}\sum_{w\in T}\lambda_{\fa,w}(x).
\end{gather*}
When $Z$ is the closed subscheme defined by $\fa$, we write these
also as $h_{Z},N_{Z},m_{Z}$ respectively. 
\end{defn}
These definitions are independent of the choice of such a field $L$.
Obviously $h_{\fa}=N_{\fa}+m_{\fa}$. Note that counting and proximity
functions depend on the choice of $S$, and their symbols are often
accompanied with the subscript $S$, which we however omit. 

That $\lambda_{Z,w}(x)$ is large means that the point $x$ is $w$-adically
close to $Z$. Thus the function $h_{Z}$ (resp. $N_{Z}$, $m_{Z}$)
expresses total closeness all over the places $w$ in $M_{L}$ (resp.
$M_{L}\setminus T$, $T$). 
\begin{defn}
If $Z=\sum_{i}c_{i}Z_{i}$ is a $\QQ$-subscheme, then we define 
\[
h_{Z}:=\sum_{i}c_{i}h_{Z_{i}},\,N_{Z}:=\sum_{i}c_{i}N_{Z_{i}},\,m_{Z}:=\sum_{i}c_{i}m_{Z_{i}}
\]
as functions on $(X\setminus\Supp(Z))(\bar{k})$. 
\end{defn}
For a (not necessarily effective) Cartier divisor $D$, the height
function $h_{D}$ on $(X\setminus\Supp(D))(\bar{k})$ defined as above
extends to the whole set $X(\bar{k})$ and defines a unique function
$h_{D}$ up to addition of bounded functions. The function class $h_{D}$
modulo bounded functions depends only on the linear equivalence class
of $D$. Furthermore, we can easily generalize this to $\QQ$-Cartier
$\QQ$-Weil divisors; if $D$ is a $\QQ$-Cartier $\QQ$-Weil divisor
and if $n$ is a positive integer such that $nD$ is Cartier, then
a height functioon $h_{D}$ is defined as $\frac{1}{n}h_{nD}$. In
particular, for a $\QQ$-Gorenstein variety $X$, we can define a
height function $h_{K_{X}}$ of a canonical divisor $K_{X}$. 
\begin{defn}
For an effective projective log pair $(X,D)$, we define $h_{K_{(X,D)}}$
as
\[
h_{K_{(X,D)}}:=h_{K_{X}}+h_{D}\colon X(\bar{k})\to(-\infty,+\infty].
\]

\end{defn}
When $D$ is a $\QQ$-Cartier $\QQ$-Weil divisor, then $h_{K_{(X,D)}}$
is a height function of $K_{X}+D$, which is considered as a ``canonical
divisor'' of $(X,D)$, hence the notation $h_{K_{(X,D)}}$. 

We need one more definition to formulate Vojta's conjecture.
\begin{defn}
For a number field $F$, let $D_{F}\in\ZZ$ denote the absolute discriminant
of $F$. For a finite extension $L/k$, we define its \emph{logarithmic
discriminant $d_{k}(L)$ }by
\[
d_{k}(L):=\frac{1}{[L:k]}\log|D_{L}|-\log|D_{k}|.
\]
For a point $x\in X(\bar{k})$ of a $k$-variety $X$, we define its
\emph{logarithmic discriminant }by
\[
d_{k}(x):=d_{k}(k(x)).
\]

\end{defn}

\section{\label{sec:conjectures}A generalization of Vojta's conjecture to
log pairs}

The original form of Vojta's conjecture for algebraic points is as
follows:
\begin{conjecture}
\label{conj:Vojta}Let $X$ be a smooth complete variety, $A$ a big
divisor of $X$ and $D$ a reduced simple normal crossing divisor
of $X$. Let $r$ be a positive integer and $\epsilon$ a positive
real number. Then there exists a proper closed subset $Z\subset X$
depending only on $X,D,A,\epsilon$ such that for all $x\in(X\setminus Z)(\bar{k})$
with $[k(x):k]\le r$, we have 
\begin{equation}
h_{K_{X}}(x)+m_{D}(x)\le d_{k}(x)+\epsilon h_{A}(x)+O(1).\label{eq:Vojta}
\end{equation}

\end{conjecture}
If we set $r=1$ in this conjecture, then $d_{k}(x)$ is always zero
and can be removed from the inequality. The conjecture in this case
is called Vojta's conjecture for rational points. 

Using log pairs and multiplier-like ideals introduced in Section \ref{sec:Multiplier-like-ideals},
we formulate a generalization of this conjecture as follows:
\begin{conjecture}
\label{conj:generalized-Vojta}Let $(X,D)$ be an effective projective
log pair, $A$ a big divisor of $X$. Let $r$ be a positive integer
and $\epsilon$ a positive real number. Then there exists a proper
closed subset $Z\subset X$ depending only on $X,D,A,\epsilon$ such
that for all $x\in(X\setminus Z)(\bar{k})$ with $[k(x):k]\le r$,
we have 
\begin{equation}
h_{K_{(X,D)}}(x)-N_{\cH(X,D)}(x)-m_{\cI(X,D)}(x)\le d_{k}(x)+\epsilon h_{A}(x)+O(1).\label{eq:general-Vojta}
\end{equation}

\end{conjecture}
We can view the left hand side of the above inequality as follows.
The main term is $h_{K_{(X,D)}}(x)$ and the other two terms are correction
terms arising from singularities of $(X,D)$. Indeed, from Proposition
\ref{prop:H-sc}, the term $-N_{\cH(X,D)}(x)$ can be thought of as
contribution of the non-sc locus $\NonSC(X,D)$. From Proposition
\ref{prop:non-lc-locus}, $-m_{\cI(X,D)}(x)$ can be thought of as
contribution of $\NonKLT(X,D)$ (and also one of $\NonLC(X,D)$ if
$X$ is log terminal outside $\Supp(D)$).
\begin{example}
If $(X,D)$ is Kawamata log terminal or if $(X,D)$ is log canonical
and $X$ is log terminal outside $\Supp(D)$, then from Proposition
\ref{prop:non-lc-locus}, we can remove the term $-m_{\cI(X,D)}(x)$.
From Corollary \ref{cor:H-def-ideal}, if we give the reduced scheme
structure to the non-sc locus $\NonSC(X,D)$, then. 
\[
N_{\cH(X,D)}(x)=N_{\NonSC(X,D)}(x).
\]
Thus inequality (\ref{eq:general-Vojta}) is written as
\begin{equation}
h_{K_{(X,D)}}(x)-N_{\NonSC(X,D)}(x)\le d_{k}(x)+\epsilon h_{A}(x)+O(1).\label{eq:Vojta-lc-case}
\end{equation}

\end{example}

\begin{example}
\label{exa:refined-Silverman}Let $X$ be a smooth projective variety
and $D$ a reduced simple normal crossing divisor on $X$. Then $\NonSC(X,D)=\Supp(D)$
and the left hand side of (\ref{eq:general-Vojta}) is equal to
\[
h_{K_{(X,D)}}(x)-N_{D}(x)=h_{K_{X}}+m_{D}.
\]
Thus Conjecture \ref{conj:generalized-Vojta} is the same as Conjecture
\ref{conj:Vojta} in this situation. 

Let $Y\subset X$ be a smooth closed subscheme of codimension $r$
whith transversally intersects $D$. Let us next consdier the log
pair $(X,D+(r-1)Y)$. Looking at the blowup of $X$ along $Y$, we
can see that this log pair is log canonical and 
\[
\NonSC(X,D+(r-1)Y)=\Supp(D).
\]
Therefore (\ref{eq:Vojta-lc-case}) in this case becomes 
\[
h_{K_{X}}(x)+h_{D}(x)+(r-1)h_{Y}(x)-N_{D}(x)\le d_{k}(x)+\epsilon h_{A}(x)+O(1).
\]
When $D$ is linearly equivalent to $-K_{X}$, then this is written
also as
\begin{equation}
h_{Y}(x)\le\frac{1}{r-1}\left(d_{k}(x)+\epsilon h_{A}(x)+N_{D}(x)\right)+O(1).\label{eq:Silverman-ineq}
\end{equation}
This slightly refines Silverman's result \cite[Th. 6]{MR2162351},
which was stated in relation to a problem of bounding greatest common
divisors, by removing $\delta\epsilon$ appearing there (cf. Appendix).
\end{example}
Although Conjecture \ref{conj:generalized-Vojta} is more general
than Conjecture \ref{conj:Vojta}, they are in fact equivalent:
\begin{prop}
\label{prop:Conjectures-equiv}Let $(X,D)$ be an effective projective
log pair and $f\colon Y\to X$ a log resolution of $(X,D)$. Suppose
that Conjecture \ref{conj:Vojta} holds for $Y$ and the reduced simple
normal crossing divisor 
\[
\lceil K_{Y/(X,D)}+\epsilon f^{*}D\rceil-\lfloor K_{Y/(X,D)}\rfloor
\]
for $0<\epsilon\ll1$. Then Conjecture \ref{conj:generalized-Vojta}
holds for $(X,D)$. In particular, Conjectures \ref{conj:Vojta} and
\ref{conj:generalized-Vojta} are equivalent.\end{prop}
\begin{proof}
The proof here is similar to the one of Vojta's similar result \cite[Prop. 4.3]{MR2867936}.
By definition, 
\begin{gather*}
f^{-1}\cH(X,D)\subset\cO_{Y}(\lfloor K_{Y/(X,D)}\rfloor)
\end{gather*}
and for $0<\epsilon\ll1$,

\[
f^{-1}\cI(X,D)\subset\cO_{Y}(\lceil K_{Y/(X,D)}+\epsilon f^{*}D\rceil).
\]
These imply 
\begin{gather*}
N_{\cH(X,D)}\circ f\ge N_{-\lfloor K_{Y/(X,D)}\rfloor},\\
m_{\cI(X,D)}\circ f\ge m_{-\lceil K_{Y/(X,D)}+\epsilon f^{*}D\rceil}.
\end{gather*}
We have 
\begin{align*}
 & (h_{K_{(X,D)}}-N_{\cH(X,D)}-m_{\cI(X,D)})\circ f\\
 & \le h_{K_{Y}}-h_{K_{Y/(X,D)}}-N_{-\lfloor K_{Y/(X,D)}\rfloor}-m_{-\lceil K_{Y/(X,D)}+\epsilon f^{*}D\rceil}\\
 & \le h_{K_{Y}}+h_{-\lfloor K_{Y/(X,D)}\rfloor}-N_{-\lfloor K_{Y/(X,D)}\rfloor}-m_{-\lceil K_{Y/(X,D)}+\epsilon f^{*}D\rceil}\\
 & =h_{K_{Y}}+m_{\lceil K_{Y/(X,D)}+\epsilon f^{*}D\rceil-\lfloor K_{Y/(X,D)}\rfloor}.
\end{align*}
To show the first assertion of the proposition, it suffices to recall
that for a big divisor $A$ on $X$, $f^{*}A$ is also big. The second
assertion is now obvious.\end{proof}
\begin{rem}
If $X$ is smooth and $D\subsetneq X$ is a genuine closed subscheme,
then $-m_{\cI(X,D)}$ is the same correction term as the one used
in \cite{MR2867936} (see Remark \ref{rem:compare-ideals}). In this
case, for a log resolution $f\colon Y\to X$ of $(X,D)$, $K_{Y/(X,D)}=K_{Y/X}-f^{-1}D$
is a $\ZZ$-divisor. Since $K_{Y/X}\ge0$,
\[
\cH(X,D)=f_{*}\cO_{Y}(K_{Y/X}-f^{-1}D)\supset f_{*}\cO_{Y}(f^{-1}D)\supset\cI_{D},
\]
where $\cI_{D}$ is the defining ideal sheaf of $D$. Therefore, 
\[
h_{K_{(X,D)}}-N_{\cH(X,D)}-m_{\cI(X,D)}\ge h_{K_{X}}+h_{D}-N_{D}-m_{\cI(X,D)}=h_{K_{X}}+m_{D}-m_{\cI(X,D)}.
\]
It follows that for such a pair $(X,D)$, Conjecture \ref{conj:generalized-Vojta}
is slightly stronger than Conjecture 4.2 of \cite{MR2867936} except
that Vojta considers more general base fields as well as non-projective
complete varieties. 
\end{rem}
Since Conjecture \ref{conj:Vojta} is known to hold if $\dim X=1$
and $r=1$ (for instances, see \cite[Rem. 14.3.5]{MR2216774} or \cite{MR1151542}),\footnote{It is said that Shinichi Mochizuki \cite{Mochizuki-IUTeich-IV} announced
a proof of Conjecture \ref{conj:Vojta} for an arbitrary $r$ in August,
2012 and it is now under a process of verification. } Proposition \ref{prop:Conjectures-equiv} implies:
\begin{cor}
\label{cor:Conj-dim-1}Suppose that $X$ has dimension one. If we
set $r=1$, then Conjecture \ref{conj:generalized-Vojta} holds.
\end{cor}

\section{\label{sec:singular-general-type}Log pairs of general type}

In this section, we specialize Vojta's conjecture to log pairs $(X,D)$
of general type and derive geometric consequences about potentially
dense subvarieties. 
\begin{defn}
A variety $X$ over $k$ is said to be \emph{potentially dense }if
for some finite extension $L/k$, $X(L)$ is Zariski dense in $X$. 
\end{defn}
For instance, an irreducible curve which is birational over $\bar{k}$
to either $\PP^{1}$ or an elliptic curve is potentially dense. More
generally, the image of a rational map $G\dasharrow X$ of a group
variety $G$ is potentially dense; this follows from the facts that
every connected group variety is an extension of an abelian variety
by a connected linear algebraic group (see \cite{zbMATH01775202}),
that every abelian variety is potentially dense \cite[Prop. 4.2]{MR2011748}
and that every linear algebraic group is unirational, in particular,
it has the Zariski dense set of rational points over any number field
\cite[Th. 18.2 and Cor. 18.3]{zbMATH00050185}. Lang \cite[p. 17]{MR1112552}
conjectured that for a smooth variety $X$ of general type (that is,
$K_{X}$ is big), there exists a proper closed subset $Z\subsetneq X$
such that every potentially dense subvariety $Y\subset X$ is contained
in $Z$. This conjecture follows from Vojta's conjecture \ref{conj:Vojta}.
We can generalize it a little to varieties with canonical singularities
as follows:
\begin{prop}
Suppose that Conjecture \ref{conj:Vojta} holds. Let $X$ be a $\QQ$-Gorenstein
canonical projective variety such that $K_{X}$ is big. Then there
exists a proper closed subset $Z\subsetneq X$ such that every potentially
dense subvariety $Y\subset X$ is contained in $Z$. \end{prop}
\begin{proof}
Let $Z\subset X$ be a proper closed subset as in Conjecture \ref{conj:Vojta}
applied to $D=0$ and $A=K_{X}$. For any fixed finite extension $L/k$
and for all $x\in(X\setminus Z)(L)$, we have
\[
(1-\epsilon)h_{A}(x)\le O(1).
\]
It follows that $h_{A}$ is bounded from above on $(X\setminus Z)(L)$.
Since $A$ is big, for any ample divisor $A'$, there exists a constant
$C>0$ such that
\[
h_{A'}(x)\le Ch_{A}(x)+O(1)
\]
for all $x\in X(\bar{k})$. Therefore $h_{A'}$ is also bounded from
above on $(X\setminus Z)(L)$. From Northcott's theorem, $(X\setminus Z)(L)$
is a finite set, which shows the proposition.
\end{proof}
It would be natural to ask what is expected for log pairs $(X,D)$
with $K_{(X,D)}=K_{X}+D$ big. First we make clear the meaning of
``$K_{X}+D$ is big''.
\begin{defn}
A $\QQ$-subscheme $D$ of a projective variety $X$ is \emph{big}
if for a log resolution $f\colon Y\to X$ of $(X,D)$, the $\QQ$-diviosr
$f^{*}D$ is big. We say that a projective log pair $(X,D)$ is \emph{of
general type} if for some (hence every) expression of $K_{X}$ as
a $\QQ$-subscheme, $K_{X}+D$ is big. \end{defn}
\begin{prop}
\label{prop:NonSC-height}Let $(X,D)$ be an effective projective
log pair of general type, let $A$ be a big divisor on $X$. Suppose
that Conjecture \ref{conj:Vojta} holds. Then there exists a proper
closed subset $Z\subset X$ and a constant $C>0$ such that for every
$x\in(X\setminus Z)(k)$, 
\[
h_{A}(x)\le Ch_{\NonSC(X,D)}(x)+O(1).
\]
Furthermore, if $(X,D)$ is log canonical, then $h_{\NonSC(X,D)}(x)$
in the above inequality can be replaced with $N_{\NonSC(X,D)}(x)$.\end{prop}
\begin{proof}
From the assumption and Proposition \ref{prop:Conjectures-equiv},
Conjecture \ref{conj:generalized-Vojta} holds. Applying it to $(X,D)$,
the given divisor $A'$ and $r=1$, we obtain the inequality 
\begin{equation}
h_{K_{(X,D)}}(x)-\epsilon h_{A'}(x)\le N_{\cH(X,D)}(x)+m_{\cI(X,D)}(x)+O(1)\label{eq:h_N-1}
\end{equation}
holding for all $x\in(X\setminus W)(k)$ for a proper closed subset
$W$. For $0<\epsilon\ll1$, we have
\begin{equation}
2\epsilon h_{A'}(x)\le h_{K_{(X,D)}}(x)+O(1)\label{eq:hA'}
\end{equation}
for $x\in(X\setminus\Supp(D))(k)$. 

If $(X,D)$ is log canonical, then $m_{\cI(X,D)}(x)=0$. In the general
case, since $\cH(X,D)\subset\cI(X,D)$, we have
\[
N_{\cH(X,D)}(x)+m_{\cI(X,D)}(x)\le N_{\cH(X,D)}(x)+m_{\cH(X,D)}(x)=h_{\cH(X,D)}(x).
\]
For an integer $C\gg0$, the $C$-th power $\cH(X,D)^{C}$ of $\cH(X,D)$
is contained in the defining ideal sheaf of the closed subset $\NonSC(X,D)$,
and 
\begin{equation}
N_{\cH(X,D)}\le CN_{\NonSC(X,D)}\text{ and }h_{\cH(X,D)}\le Ch_{\NonSC(X,D)}.\label{eq:NCN}
\end{equation}
Combining (\ref{eq:h_N-1}), (\ref{eq:hA'}) and (\ref{eq:NCN}),
we obtain the proposition. 
\end{proof}
We may interpret this result as that most rational points lie near
$\NonSC(X,D)$, which sounds a little surprising, in particular, if
$\NonSC(X,D)$ has codimension $\ge2$. 
\begin{cor}
\label{cor:singular-Lang-NonSC}Let $(X,D)$ be an effective log pair
of general type. Suppose that Conjecture \ref{conj:Vojta} holds.
Then there exists a proper closed subset $Z\subset X$ such that for
every potentially dense closed subvariety $Y\subset X$, either $Y$
is contained in $Z$, or $Y$ intersects $\NonSC(X,D)$. \end{cor}
\begin{proof}
Let $Z\subset X$ be a closed subvariety as in Proposition \ref{prop:NonSC-height}.
To obtain a contradiction, supppose that there exists a potentially
dense closed subvariety $Y\subset X$ such that $Y\cap\NonSC(X,D)=\emptyset$
and $Y\not\subset Z$. If necessary enlarging $k$, we may suppose
that $Y(k)$ is Zariski dense in $Y$. From the lemma below, the height
function $h_{\NonSC(X,D)}$ is bounded on $Y(k)$. On the other hand,
from Proposition \ref{prop:NonSC-height} and Northcott's theorem,
the same function cannot be bounded on $Y(k)$, a contradition. The
corollary follows. \end{proof}
\begin{lem}
\label{lem:disjoint-height-bounded}Let $X$ be a projective variety
and $C,D\subset X$ proper closed subschemes with $C\cap D=\emptyset$.
Let $h_{D}\colon X(k)\to\RR\cup\{\infty\}$ be a height function of
$D$. Then its restriction $h_{D}|_{C(k)}$ is a bounded function.\end{lem}
\begin{proof}
From the functoriality of the Weil function, $h_{D}|_{C(k)}$ is a
height function of $D\cap C$ as a closed subscheme of $C$. In our
situation, it is empty and any height function of it is bounded. 
\end{proof}
In Section \ref{sec:Rational-surfaces-of-general-type}, we construct
rational projective surfaces of general type. If we restrict ourselves
to smooth projective varieties, being rational and being of general
type are completely opposite properties and cannot hold at the same
time. However, they can be compatible with each other for singular
varieties. 
\begin{defn}
An irreducible variety $X$ over $\bar{k}$ is said to be \emph{rationally
connected} if general two points $x,y\in X$ are connected by a rational
curve. 
\end{defn}
Rational varieties are rationally connected. The next proposition
shows that the assertions of Proposition \ref{prop:NonSC-height}
and Corollary \ref{cor:singular-Lang-NonSC} never hold if $X$ is
rationally connected and $\NonSC(X,D)$ is replaced with a smooth
closed subset of codimension $\ge2$ which is contained in the smooth
locus of $X$. 
\begin{prop}
\label{prop:rationally-connected}Let $X$ be a rationally connected
projective variety over $\bar{k}$, let $Z\subsetneq X$ be a proper
closed subset and let $W\subset X$ be a smooth closed subvariety
of codimension $\ge2$ with $W\subset X_{\sm}$. Then there exists
a rational curve $C\subset X$ such that $C\not\not\subset Z$ and
$C\cap W=\emptyset$. \end{prop}
\begin{proof}
Taking a resolution, we may suppose that $X$ is smooth. Let $x\in X(\bar{k)}$
be a point outside $Z\cup W$. By \cite[2.1]{MR1158625}, there exists
a morphism
\[
f\colon\PP^{1}\to X
\]
such that
\begin{enumerate}
\item $f$ is an immersion in the sense that for every closed point $p\in\PP^{1}$,
$T_{p}\PP^{1}\to T_{f(p)}X$ is injective, 
\item $x\in f(\PP^{1})$, 
\item $f^{*}T_{X}$ is ample.
\end{enumerate}
Since 
\[
\mathrm{Ext}^{1}(f^{*}\Omega_{X},\cO_{\PP^{1}})=H^{1}(\PP^{1},f^{*}T_{X})=0,
\]
from \cite[Ch. I, Th. 2.16]{MR1440180}, the moduli space of morphisms
$\PP^{1}\to X$, $\Hom(\PP^{1},X)$, is smooth at $[f]$ and the tangent
space $T_{[f]}\Hom(\PP^{1},Y)$ is identified with $H^{0}(\PP^{1},f^{*}T_{Y})$. 

Let $B:=f^{-1}W$. If $B$ is empty, then $f(\PP^{1})$ is a desired
rational curve. Otherwise, there exists a subvariety $V\subset H^{0}(\PP^{1},f^{*}T_{X})$
of dimension $\dim X-1$ such that for every $b\in B$, the image
of $V$ in $T_{f(b)}X$ is transversal to the image of $T_{b}\PP^{1}$.
We take a smooth irreducible subvariety $W\subset\Hom(\PP^{1},Y)$
of dimension $\dim X-1$ passing through $[f]$ such that $T_{[f]}W=V$.
The induced morphism
\[
G\colon\PP^{1}\times W\to Y
\]
is then an immersion around $B\times\{[g]\}$. Therefore $G^{-1}W$
has codimension $\ge2$ in $\PP^{1}\times W$ around $\PP^{1}\times\{[g]\}$.
This shows that $G^{-1}W$ does not surject onto $W$. We conclude
from this that for a general point $[h]\in V$, the image of $h\colon\PP^{1}\to X$
does not intersect $W$. Similarly, $G^{-1}Z$ has codimension $\ge1$.
Hence, for a general point $[h]\in V$, the fiber $\PP^{1}\times\{[h]\}$
is not contained in $G^{-1}Z$. This means that the image of $h\colon\PP^{1}\to Y$
is not containd in $Z$. Thus, for a general $[h]\in W$, $h(\PP^{1})$
is a rational curve satisfying the desired condition. \end{proof}
\begin{example}
The arguably easiest way to construct a rationally connected variety
of general type is to take the closure of the image of a morphism
$\AA_{k}^{n}\to\PP_{k}^{n+1}$. For an irreducible polynomial $f(x_{1},\dots,x_{n})$
of degree $d$, we consider a closed embedding 
\[
g\colon\AA_{k}^{n}\to\AA_{k}^{n+1},\,(x_{1},\dots,x_{n})\mapsto(x_{1},\dots,x_{n},f(x_{1},\dots,x_{n})).
\]
The closure $X$ of $g(\AA_{k}^{n})$ in $\PP_{k}^{n+1}$ is a hypersurface
of degree $d$. Threfore $X$ is Gorenstein and rational. By the adjunction
formula, if $d\ge n+2$, then $X$ is also of general type. However
the author does not know what kind of singularities such an $X$ would
have. In the next section, we will construct a rational surface of
general type having only log terminal singularities.
\end{example}

\section{\label{sec:Rational-surfaces-of-general-type}Rational surfaces of
general type}

In this section, we show the following proposition.
\begin{prop}
\label{prop:singular-rational-surf}Suppose that $k$ contains a primitive
$n$-th root for some $n\ge5$. Then there exists a projective rational
surface $X$ having only log terminal singularities such that $K_{X}$
is ample.\end{prop}
\begin{proof}
We construct such an $X$ as the quotient of a Fermat hypersurface
by a finite group action. Let $n\ge5$ be an integer such that $k$
contains a primitive $n$-th root $\zeta$ of $1$. Consider the Fermat
hypersurface of degree $n$,
\[
F:=V(x_{0}^{n}+x_{1}^{n}+x_{2}^{n}+x_{3}^{n})\subset\PP_{k}^{3}.
\]
This is a smooth irreducible projective surface. From the condition
$n\ge5$ and the adjunction formula, the canonical divisor $K_{F}$
is ample. Let $G=\langle g\rangle=\ZZ/n\ZZ$ be a cyclic group of
order $n$ generated by an element $g$. We define a $G$-action on
$\PP^{3}$ by 
\[
g((x_{0}:x_{1}:x_{2}:x_{3}))=(\zeta x_{0}:\zeta x_{1}:x_{2}:x_{3}).
\]
Clearly $F$ is preserved by the action. We see that the fixed point
locus $F^{G}$ in $F$ is 
\[
\{x_{0}=x_{1}=x_{2}^{n}+x_{3}^{n}=0\}\cup\{x_{2}=x_{3}=x_{0}^{n}+x_{1}^{n}=0\},
\]
which has dimesion zero and $2n$ points. Let $X:=F/G$ be the associated
quotient variety and $\pi\colon F\to X$ the natural morphism, which
is étale in codimension one. Since $K_{F}=\pi^{*}K_{X}$ and $K_{F}$
is ample, $K_{X}$ is also ample. The variety $X$ has only quotient
singularities. As is well known, quotient singularities are log terminal
(for instance, see \cite[Cor. 2.43 or p. 103]{MR3057950}).

It remains to show that $X$ is rational. Consider the following locally
closed subvariety of $\AA_{k}^{3}$ with coordinates $x,y,z$, 
\[
W=\left\{ x^{n}+1\ne0,\,z^{n}+\frac{y^{n}+1}{x^{n}+1}=0\right\} \subset\AA_{k}^{3},
\]
and the morphism 
\begin{align*}
\phi\colon W & \to\PP_{k}^{3}\\
(x,y,z) & \mapsto(xz:z:y:1).
\end{align*}
We see that its image is contained in $F\cap\{x_{3}\ne0\}$; we denote
the induced morphism $W\to F$ by $\psi$. The ring homomorphism associated
to 
\[
W\to\AA_{k}^{3}=\{x_{3}\ne0\}
\]
is given by 
\begin{align*}
\alpha\colon k[u_{0},u_{1},u_{2}] & \to\frac{k\left[x,y,z,\frac{1}{x^{n}+1}\right]}{\left\langle z^{n}+\frac{y^{n}+1}{x^{n}+1}\right\rangle }\\
u_{0} & \mapsto xz\\
u_{1} & \mapsto z\\
u_{2} & \mapsto y.
\end{align*}
Since the image of $\alpha$ together with $k$ generates the function
field of $W$ as a field, the morphism $\psi\colon W\to F$ is birational. 

We define a $G$-action on $W$ by 
\[
g\colon(x,y,z)\mapsto(x,y,\zeta z).
\]
Then $\psi$ is $G$-equivariant. Taking quotients, we otain a morphism
\[
\bar{\psi}\colon W/G\to X,
\]
which is again birational. The variety $W/G$ is naturally isomorphic
to the open subvariety $\{x^{n}+1\ne0\}$ of $\AA_{k}^{2}$. In particular,
it is rational. We conclude that $X$ is also rational. We have proved
Proposition \ref{prop:singular-rational-surf}.\end{proof}
\begin{rem}
For a suitable local coordinates $r,s$ around each point of $F^{G}$,
the group action is given by 
\[
g(r,s)=(\zeta r,\zeta s)\text{ or }(\zeta^{-1}r,\zeta^{-1}s).
\]
Namely every singular point of $X$ is the quotient singularity of
type $\frac{1}{n}(1,1)$. It follows that 
\[
\totaldiscrep(X)=\discrep(X)=\frac{2}{n}-1<0
\]
(for instance, see \cite[Cor. 6]{MR2271984}). In particular, $X$
is log terminal, but not canonical. Furthermore, the singular locus
of $X$ coincides with the non-canonical locus $\NonC(X)$. 
\end{rem}

\begin{rem}
The above proposition shows the necessity of the correction term $-N_{\cH(X,D)}$
in Conjecture \ref{conj:generalized-Vojta}. Indeed, for such an $X$,
if we set $D=0$, $r=1$ and $A=K_{X}$, then inequality (\ref{eq:general-Vojta})
is written as 
\begin{equation}
(1-\epsilon)h_{K_{X}}(x)-N_{\NonC(X)}(x)\le O(1).\label{eq:ineq-rat-surf}
\end{equation}
Since $X$ is rational, the $k$-point set $X(k)$ is Zariski dense.
If there was no correction term $N_{\NonC(X)}$, this fact would contradicts
Northcott's theorem. 
\end{rem}

\begin{rem}
From Proposition \ref{prop:rationally-connected}, inequality (\ref{eq:ineq-rat-surf})
does not hold if we replace $\NonC(X)$ with any collection of finitely
many smooth points and if we replace $k$ with some finite extension
of it. 
\end{rem}

\begin{rem}
From Corollary \ref{cor:singular-Lang-NonSC}, if Vojta's conjecture
is true, then all but finitely many irreducible curves on $X\otimes_{k}\bar{k}$
birational to $\PP^{1}$ or an elliptic curve would pass through one
of the singular points of $X$. 
\end{rem}
\appendix

\section{Greatest common divisors and plane curves}

Bugeaud, Corvaja and Zannier \cite{MR1953049,MR2130274} obtained
an upper bound for $\gcd(a-1,b-1)$ for certain families of integer
pairs $(a,b)$. To explain their result in relation to Vojta's conjecture,
Silverman \cite{MR2162351} observed that the greatest common divisor
is essentially a height function associated to a subscheme of codimension
$\ge2$, although he uses the blowup along the subscheme and a height
function associated to the exceptional divisor instead (see also \cite{MR2846352,MR2794199}).
He then formulated a conjectural generalization of the result of Bugeaud,
Corvaja and Zannier. It was in this work that a slightly weaker version
of inequality (\ref{eq:Silverman-ineq}) appeared.

In this Appendix, as an application of Silverman's observation, we
relate estimation of $\gcd(a,b)$ for integer pairs $(a,b)$ satisfying
an algebraic equation with the multiplicity of the corresponding plane
curve at the origin. The only ingredients necessary to do so is basic
properties of heights and a simple analysis of resolution of curves.
\begin{lem}
Let $Z\subset\PP_{\QQ}^{n}$ be the closed subscheme defined by the
ideal $\langle f_{1},\dots,f_{l}\rangle\subset\QQ[x_{0},\dots,x_{n}]$
generated by homogenous polynomials $f_{1},\dots,f_{l}\in\ZZ[x_{0},\dots,x_{n}]$.
For a point $x\in\PP_{\QQ}^{n}(\QQ)$, we write $x=(x_{0}:x_{1}:\cdots:x_{n})$
in terms of integers $x_{i}$ with $\gcd(x_{0},x_{1},\dots,x_{n})=1$
and define $f_{i}(x):=f_{i}(x_{0},\dots,x_{n})\in\ZZ$. Then
\[
N_{Z}(x):=\log\gcd(f_{1}(x),\dots,f_{l}(x))
\]
is a counting function of $Z$ with respect to $S=\{\infty\}$, and 

\[
h_{Z}(x):=\log\gcd(f_{1}(x),\dots,f_{l}(x))-\max_{1\le i\le l}\log\frac{|f_{i}(x)|_{\infty}}{\max\{|x_{0}|_{\infty},\dots,|x_{n}|_{\infty}\}^{\deg f_{i}}}
\]
is a height function of $Z$.\end{lem}
\begin{proof}
We first note that for integers $a_{i}$, 
\[
\log\gcd(a_{1},\dots,a_{l})=-\sum_{p\in M_{\QQ};\,p\ne\infty}\max_{i}|a_{i}|_{p}.
\]
From Propositions \ref{prop:properties-Weil-fn} and \ref{prop:normalization-Weil-fn},
\begin{align*}
\lambda_{Z,p}(x): & =\min_{1\le i\le l}\left\{ -\log\frac{|f_{i}(x)|_{p}}{\max\{|x_{0}|_{p},\dots,|x_{n}|_{p}\}^{\deg f_{i}}}\right\} \quad(p\in M_{\QQ})
\end{align*}
is a Weil function of $Z$. For $p\ne\infty$, since $\gcd(x_{0},\dots,x_{n})=1$,
we have
\[
\max\{|x_{0}|_{p},\dots,|x_{n}|_{p}\}=1
\]
and 
\[
\lambda_{Z,p}(x)=-\log\max_{i}|f_{i}(x)|_{p}.
\]
We conclude that
\begin{align*}
N_{Z}(x): & =\sum_{p\in M_{\QQ};\,p\ne\infty}\lambda_{Z,p}(x)\\
 & =-\sum_{p\in M_{\QQ};\,p\ne\infty}\log\max_{i}|f_{i}(x)|_{p}\\
 & =\log\gcd(f_{1}(x),\dots,f_{l}(x))
\end{align*}
is a counting function of $Z$ and
\begin{align*}
h_{Z}(x) & :=N_{Z}(x)+\lambda_{Z,\infty}(x)\\
 & =\log\gcd(f_{1}(x),\dots,f_{l}(x))-\max_{1\le i\le l}\log\frac{|f_{i}(x)|_{\infty}}{\max\{|x_{0}|_{\infty},\dots,|x_{n}|_{\infty}\}^{\deg f_{i}}}
\end{align*}
is a height function of $Z$. \end{proof}
\begin{example}
For $l<n$, let $Z$ be the linear subspace defined by 
\[
x_{0}=x_{1}=\cdots=x_{l}=0.
\]
Then
\[
N_{Z}(x)=\log\gcd(x_{0},\dots,x_{l})
\]
is a counting function of $Z$. Since 
\begin{align*}
 & \max_{0\le i\le l}\frac{|x_{i}|_{\infty}}{\max\{|x_{0}|_{\infty},\dots,|x_{n}|_{\infty}\}}\\
 & =\frac{\max\{|x_{0}|_{\infty},\dots,|x_{l}|_{\infty}\}}{\max\{|x_{0}|_{\infty},\dots,|x_{n}|_{\infty}\}}\\
 & =\min\left\{ 1,\frac{\max\{|x_{0}|_{\infty},\dots,|x_{l}|_{\infty}\}}{\max\{|x_{l+1}|_{\infty},\dots,|x_{n}|_{\infty}\}}\right\} ,
\end{align*}
the function 
\[
h_{Z}(x)=\log\gcd(x_{0},\dots,x_{l})-\log\min\left\{ 1,\frac{\max\{|x_{0}|_{\infty},\dots,|x_{l}|_{\infty}\}}{\max\{|x_{l+1}|_{\infty},\dots,|x_{n}|_{\infty}\}}\right\} 
\]
is a height function of $Z$. \end{example}
\begin{lem}
\label{lem:curve-ht}Let $X$ be an irreducible projective variety
of dimension one over a number field $k$ and $\pi\colon\tilde{X}\to X$
the normalization. Let $Z\subset X$ be a proper closed subscheme
and $l\in\ZZ$ the degree of the scheme-theoretic pull-back $\pi^{-1}Z$
naturally regarded as a divisor. Let $D$ be a divisor of $X$ of
degree $l$ supported in the smooth locus of $X$. Then, for every
$\epsilon>0$, their exist constants $C_{1},C_{2}>0$ such that for
all $x\in(X\setminus Z)(\bar{k})$, 
\begin{equation}
(1-\epsilon)h_{D}(x)-C_{1}\le h_{Z}(x)\le(1+\epsilon)h_{D}(x)+C_{2}.\label{eq:quasi-eq}
\end{equation}
Moreover, if $X$ is rational (that is, birational to $\PP_{k}^{1}$),
then 
\[
h_{Z}(x)=h_{D}(x)+O(1).
\]
 \end{lem}
\begin{proof}
Let $\tilde{Z}:=\pi^{-1}Z$ and $\tilde{D}:=\pi^{*}D$. Since they
are divisors of equal degree, height functions $h_{\tilde{Z}}$ and
$h_{\tilde{D}}$ are quasi-equivalent (see \cite[Cor. 3.5, Ch. 4]{MR715605}),
hence so are $h_{D}$ and $h_{Z}$; it exactly means (\ref{eq:quasi-eq}).
If $X$ is rational, then $\tilde{Z}$ and $\tilde{D}$ are linearly
equivalent. Therefore $h_{\tilde{Z}}$ and $h_{\tilde{D}}$ differs
only by a bounded function, and the same holds for $h_{Z}$ and $h_{D}$.\end{proof}
\begin{thm}
\label{thm:h_O}Let $X\subset\PP_{k}^{2}$ be an integral plane curve
of degree $d$ and let $O:=(0:0:1)\in\PP_{k}^{2}(k)$. Suppose that
$X$ has multiplicity $m$ at $O$, that is, $m$ is the largest integer
$n$ such that $\cI_{X,O}\subset\fm_{O}^{n}$, where $\cI_{X}\subset\cO_{\PP_{k}^{2}}$
is the defining ideal sheaf of $X$, $\cI_{X,O}$ is its stalk at
$O$ and $\fm_{O}$ is the maximal ideal of the local ring $\cO_{\PP_{k}^{2},O}$.
Let $h$ be the standard logarithmic height on $\PP_{k}^{2}$ given
by
\[
h((x:y:z))=\sum_{w\in M_{L}}\log\max\{\Vert x\Vert_{w},\Vert y\Vert_{w},\Vert z\Vert_{w}\}
\]
for so large finite extention $L/k$ that $x,y,z\in L$. Then, for
every $\epsilon>0$, their exist constants $C_{1},C_{2}>0$ such that
for all $x\in(X\setminus\{O\})(\bar{k})$, 
\[
\left(\frac{m}{d}-\epsilon\right)h(x)-C_{1}\le h_{O}(x)\le\left(\frac{m}{d}+\epsilon\right)h(x)+C_{2}.
\]
Moreover, if $X$ is rational, then 
\[
h_{O}(x)=\frac{m}{d}h(x)+O(1).
\]
\end{thm}
\begin{proof}
The standard height $h$ is a height function of a line in $\PP_{k}^{2}$.
Take a general line $L$ which does not meet any singularity of $X$.
We regard the closed point $O$ as a reduced scheme and apply Lemma
\ref{lem:curve-ht} to $Z=O$ and $D=L\cap X$. To see the assertion,
we need to show that $m$ is equal to $l$ as in Lemma \ref{lem:curve-ht}.
Since these numbers are stable under extension of the base field,
we consider a plane curve germ $\hat{X}=\Spec\bar{k}[[x,y]]/\langle f\rangle$
defined over $\bar{k}$. The multiplicity is then equal to the order
of $f$. If $\hat{X}_{i}$, $i=1,\dots,r$, are the irreducible components
of $\hat{X}$ and if $m_{i}$ and $l_{i}$ are the numbers similarly
defined for $\hat{X_{i}}$, then 
\[
m=\sum_{i=1}^{r}m_{i}\text{ and }l=\sum_{i=1}^{r}l_{i}.
\]
Therefore, we may assume that $\hat{X}$ is irreducible. Then $\hat{X}\cong\Spec\bar{k}[[g,h]]$,
where $g,h\in\bar{k}[[t]]$ are power series of distinct orders such
that $\Spec\bar{k}[[t]]\to\hat{X}$ is birational. Now it is easy
to see that
\[
m=\min\{\ord(g),\ord(f)\}=l.
\]
We have completed the proof. 
\end{proof}
Note that the theorem is valid even if $O\notin X$; then $m=0$ and
$h_{O}$ is bounded (Lemma \ref{lem:disjoint-height-bounded}). The
theorem asserts that a singular point has more rational points around
it more than a smooth point does and that its extent is determined
by the multiplicity, the most fundamental invariant of plane curve
singularities.
\begin{rem}
Theorem \ref{thm:h_O} is non-trivial only when $X$ has infinitely
many $k$-points; it means from Faltings' theorem that $X$ has a
geometric irreducible component birational to $\PP^{1}$ or an elliptic
curve. If $X$ is smooth, then this is possible only when $d\le3$.
However, if we allow singularities, then there exist plane curves
of arbitrary degree having infinitely many $k$-points.
\end{rem}
Specializing the theorem to the case $k=\QQ$ and to $\QQ$-rational
points, we obtain:
\begin{cor}
Let $f(x,y)\in\QQ[x,y]$ be an irreducible polynomial and let $d$
and $m$ be the degree and the order of $f$ respectively. Then, for
every $\epsilon>0$, their exist positive constants $C_{1},C_{2}$
such that for all triplets $(x,y,z)\ne(0,0,0),\,(0,0,1)$ of integers
satisfying $\gcd(x,y,z)=1$ and $f(x,y,z)=0$, we have 
\begin{multline}
\left(\frac{m}{d}-\epsilon\right)\log\max\{|x|,|y|,|z|\}-C_{1}\le\log\gcd(x,y)-\log\min\left\{ 1,\frac{\max\{|x|,|y|\}}{|z|}\right\} \\
\le\left(\frac{m}{d}+\epsilon\right)\log\max\{|x|,|y|,|z|\}+C_{2}.\label{eq:gcd}
\end{multline}
Moreover, if $X$ is rational, then
\[
\log\gcd(x,y)-\log\min\left\{ 1,\frac{\max\{|x|,|y|\}}{|z|}\right\} =\frac{m}{d}\log\max\{|x|,|y|,|z|\}+O(1).
\]

\end{cor}
Furthermore, excluding points close to the origin relative to the
Euclidean topology, we obtain the following simpler estimation. 
\begin{cor}
\label{cor:gcd}With the same notation as above, for every $\epsilon,\delta>0$,
their exist positive constants $C_{1}',C_{2}'$ such that for all
triplets $(x,y,z)\ne(0,0,0),\,(0,0,1)$ of integers satisfying $\gcd(x,y,z)=1$,
$f(x,y,z)=0$ and $\max\{|x/z|,|y/z|\}\ge\delta$, we have 
\[
C_{1}'\max\{|x|,|y|\}^{m/d-\epsilon}\le\gcd(x,y)\le C_{2}'\max\{|x|,|y|\}^{m/d+\epsilon}.
\]
Moreover, if $X$ is rational, then we can replace $\epsilon$ with
zero. \end{cor}
\begin{proof}
From the condition $\max\{|x/z|,|y/z|\}\ge\delta$, the term 
\[
-\log\min\left\{ 1,\frac{\max\{|x|,|y|\}}{|z|}\right\} 
\]
in (\ref{eq:gcd}) is bounded and hence can be eliminated. If $\delta\ge1$,
then the condition $\max\{|x/z|,|y/z|\}\ge\delta$ implies 
\[
\log\max\{|x|,|y|,|z|\}-\log\max\{|x|,|y|\}=0.
\]
If $\delta<1$, then 
\begin{align*}
0\le & \log\max\{|x|,|y|,|z|\}-\log\max\{|x|,|y|\}\\
 & \le-\log\max\{|x/z|,|y/z|\}\le-\log\delta.
\end{align*}
Therefore $\log\max\{|x|,|y|,|z|\}$ in (\ref{eq:gcd}) can be replaced
with $\log\max\{|x|,|y|\}$. Writing the resulting inequalities mutliplicatively,
we obtain the corollary.
\end{proof}
Note that the condition imposed in the last corollary on triplets
$(x,y,z)$ are satisfied by $(x,y,1)$ for integer pairs $(x,y)$
with $f(x,y)=0$.
\begin{example}
Let $X\subset\AA_{\QQ}^{2}$ be the affine plane curve defined by
$x^{d}=y^{m}$ for coprime positive integers $d,m$ with $d>m$. This
curve is rational and has degree $d$ and multiplicity $m$ at $O$.
An integral point $p$ of $X$ is of the form $(a^{m},a^{d})$ for
an integer $a$. With $O=(0,0)$, we have 
\[
\gcd(a^{m},a^{d})=|a^{m}|=\max\{|a^{m}|,|a^{d}|\}^{m/d}.
\]

Next consider the affine plane curve $Y$ defined by $(x+1)^{d}=(y+1)^{m}$
for the same $d,m$ as above. This is a translation of $X$. Note
that $Y$ contains $O$ as a smooth point, namely $Y$ has multiplicity
one at $O$. An integral point $p$ of $Y$ is of the form $(a^{m}-1,a^{d}-1)$
for an integer $a$. We claim that for $|a|>1$, 
\[
\gcd(a^{m}-1,a^{d}-1)=|a-1|.
\]
To show this, we need to show that 
\[
\gcd(a^{m-1}+a^{m-1}+\cdots+1,a^{d-1}+a^{d-1}+\cdots+1)=1,
\]
which can be proved by induction and using the fact that
\begin{align*}
 & \gcd(a^{m-1}+a^{m-1}+\cdots+1,a^{d-1}+a^{d-1}+\cdots+1)\\
 & =\gcd(a^{m-1}+a^{m-1}+\cdots+1,a^{(d-m)-1}+a^{(d-m)-1}+\cdots+1).
\end{align*}
From the claim, 
\[
\gcd(a^{m}-1,a^{d}-1)\sim\max\{|a^{m}-1|,|a^{d}-1|\}^{1/d}\quad(|a|\to\infty).
\]

Finally consider the curve $Z$ defined by $x^{d}=(y+1)^{m}$. This
curve does not pass through the origin, equivalently it has multiplicity
$m=0$ at $O$. An integral point $p$ of $Z$ is of the form $(a^{m},a^{d}-1)$
for an integer $a$. Clearly 
\[
\gcd(a^{m},a^{d}-1)=1=\max\{|a^{m}|,|a^{d}-1|\}^{0/d}.
\]

\end{example}
\bibliographystyle{alpha}
\bibliography{../mybib}

\end{document}